\newif\ifentcs
\entcsfalse
\ifentcs
\documentclass[9pt]{entcs}
\usepackage{entcsmacro}
\else
\documentclass{article}
\usepackage{amssymb,hyperref}

\newtheorem{theorem}{Theorem}[section]
\newtheorem{proposition}[theorem]{Proposition}
\newtheorem{corollary}[theorem]{Corollary}
\newtheorem{lemma}[theorem]{Lemma}

\newenvironment{proof}{\emph{Proof.}}{\relax\hfill\qed}
\newtheorem{definition}[theorem]{Definition}
\newtheorem{remark}[theorem]{Remark}
\newtheorem{example}[theorem]{Example}
\newcommand\qed{\hfill$\Box$\vskip0.2em}
\fi
\usepackage{url}
\usepackage{latexsym}
\usepackage{stmaryrd,amsmath,amsfonts}
\usepackage{graphicx}
\usepackage[all]{xypic}
\newcommand{\exr}{\overline {\mathbb R}_{+}}
\newcommand{\set}[2]{\{#1\mid#2\}}

\newcommand{\id}{\mathrm{id}}
\newcommand{\TOP}{{\bf TOP$_0$}}
\newcommand{\V}{{\mathcal V_w}}
\newcommand\simple{{\mathrm f}}
\newcommand{\VF}{{\mathcal V_\simple}}
\newcommand{\VP}{{\mathcal V_{\mathrm p}}}
\newcommand{\ctop}{\bf TOP_{0}}

\newcommand{\real}{\mathbb{R}}
\newcommand\Rplus{\real_+}
\newcommand\creal{\overline\real_+}

\newcommand\Lform{\mathcal L}

\newcommand{\identity}[1]{\mathrm{id}_{#1}}

\newcommand\nat{\mathbb{N}}

\newcommand\upc{\mathop{\uparrow}}
\newcommand\uuarrow{\rlap{$\uparrow$}\raise.5ex\hbox{$\uparrow$}}
\newcommand\dc{\mathop{\downarrow}}
\newcommand\ddarrow{\rlap{$\downarrow$}\raise.5ex\hbox{$\downarrow$}}

\newcommand\limp{\mathrel{\Rightarrow}}

\newcommand\Open{{\mathcal O}}

\newcommand\Sierp{\mathbb{S}}
\newcommand\App{\mathsf{App}}
\newcommand\supp{\mathop{\text{supp}}}

 \newcommand\ForAuthors[1]
 {\par\smallskip                     
  \begin{center}
   \fbox
   {\parbox{0.9\linewidth}
    {\raggedright\sc--- #1}
   }
  \end{center}
  \par\smallskip                     %
 }

 \newcommand\papertitle{Algebras of the extended probabilistic powerdomain monad}
 \newcommand\paperabs{%
    We investigate the Eilenberg-Moore algebras of the extended
    probabilistic powerdomain monad~$\V$ over the category $\ctop$ of
    $T_0$ topological spaces and continuous maps. We prove that every
    $\V$-algebra in our setting is a weakly locally convex sober
    topological cone, and that a map is the structure map of a
    $\V$-algebra if and only if it is continuous and sends every
    continuous valuation to its unique barycentre.  Conversely, for
    locally linear sober cones---a strong form of local convexity---,
    the mere existence of barycentres entails that the barycentre map
    is the structure map of a $\V$-algebra; moreover the algebra
    morphisms are exactly the linear continuous maps in that case.
   
    We also examine the algebras of two related monads, the simple
    valuation monad~$\VF$ and the point-continuous valuation
    monad~$\VP$. In~$\ctop$ their algebras are fully characterised as
    weakly locally convex topological cones and weakly locally convex
    sober topological cones, respectively. In both cases, the algebra
    morphisms are continuous linear maps between the corresponding
    algebras.
  }
  \newcommand\digicosme{%
    This research was partially supported by Labex
    DigiCosme (project ANR-11-LABEX-0045-DIGICOSME) operated by ANR as
    part of the program ``Investissement d'Avenir'' Idex Paris-Saclay
    (ANR-11-IDEX-0003-02).
  }
  \newcommand\lsv{%
    LSV, ENS Paris-Saclay, CNRS, Universit\'e Paris-Saclay,
    France
    }
 \ifentcs
\begin{document}
\begin{frontmatter}
  \title{\papertitle}
  \author[]{Jean Goubault-Larrecq 
    \thanksref{Digicosme}\thanksref{myemail}}
  \author[]{Xiaodong Jia
    \thanksref{Digicosme}\thanksref{xjemail}}
  \address{\lsv}
  \thanks[Digicosme]{\digicosme}
  \thanks[myemail]{Email:
    \href{mailto: goubault@ens-paris-saclay.fr} {\texttt{\normalshape
        goubault@ens-paris-saclay.fr}}}
  \thanks[xjemail]{Email:
    \href{mailto: jia@lsv.fr} {\texttt{\normalshape jia@lsv.fr}}}
  \begin{abstract}
    \paperabs
  \end{abstract}
  \begin{keyword}
    Valuation powerdomain, monads, Eilenberg-Moore algebras, topological cones
  \end{keyword}
\end{frontmatter}
\else
\title{\papertitle}
\author{Jean Goubault-Larrecq \\
  \url{goubault@ens-paris-saclay.fr} \\[0.5ex]
  Xiaodong Jia\thanks{\digicosme} \\
  \url{jia@lsv.fr} \\[0.5ex]
  \lsv
  }
\begin{document}
\maketitle

\begin{abstract}
  \paperabs
\end{abstract}
\fi

\section{Introduction}
\label{sec:intro}

The \emph{probabilistic powerdomain construction} on directed complete
partially ordered sets (dcpo for short) was introduced by Jones and
Plotkin and employed to give semantics to programming languages with
probabilistic features~\cite{jones89,jones90}. The probabilistic
powerdomain of a dcpo consists of \emph{continuous valuations} defined
on the Scott-opens of the dcpos, where a valuation is a function
assigning real numbers to Scott-open subsets of the dcpo.  Jones
proved that this construction is a monad on the category of dcpos and
Scott-continuous functions. Moreover, she proved that this monad can
be restricted to the full subcategory of continuous domains, the
algebras of this monad in the category of continuous domains are the
\emph{continuous abstract probabilistic domains}, and the algebra
homomorphisms are continuous linear maps. Kirch~\cite{kirch93}
generalised Jones and Plotkin's probabilistic powerdomain by
stipulating that a valuation might take values that are not finite.
He showed that this construction is again a monad that can be
restricted to the category of continuous domains, and the algebras of
this monad in the category of continuous domains are the
\emph{continuous d-cones}, a notion well investigated 
in~\cite{tix99}.

A topological counterpart of the probability powerdomain construction
was considered by Heckmann in~\cite{heckmann96} and then by
Alvarez-Manilla, Jung and Keimel in~\cite{jung04,alvarez04}. They
considered the \emph{weak topology} on the set of continuous
valuations instead of the Scott topology. In~\cite[Proposition
5.1]{heckmann96}, Heckmann proved that the resulting space is sober
for any topological spaces; and in~\cite{jung04,alvarez04}, the
authors proved that the resulting topological space is stably compact
if the underlying space is stably compact. This topological
construction is consistent with earlier work~\cite{kirch93}, where
Kirch proved that the weak topology and Scott topology coincide on the
set of continuous valuations if one starts with a continuous
domain. Cohen, Escard\'o 
and Keimel further developed this construction in~\cite{cohen06},
where they employed the theory of \emph{topological cones} to retrieve
the definition and called the construction the \emph{extended
  probabilistic powerdomain} over $T_0$ spaces. They showed that the
extended probabilistic powerdomain construction is a monad over the
category of $T_0$ topological spaces and considered its algebras in
related categories in the same paper, leaving a conjecture that the
algebras of this monad on the category of stably compact spaces and
continuous functions are the stably compact locally convex topological
cones. Restricting this monad to the category of compact ordered
spaces (compact pospaces) and continuous monotone maps, Keimel located
the algebras of this monad to be the compact convex ordered sets
embeddable in locally convex ordered topological vector
spaces~\cite{keimel08a}.

\subsection*{Outline.}

We are concerned about the algebras of the extended probabilistic
powerdomain in the category of $T_0$ topological spaces and continuous
functions.  We recall some known facts about the extended
probabilistic powerdomain monad in Section~\ref{sec:V}, and on cones
in Section~\ref{cones}.
We prove in Section~\ref{algebras} that every algebra of this monad in
the category of $T_0$ spaces is a \emph{weakly locally convex sober
  topological cone}, and algebra morphisms must be continuous
\emph{linear maps}.  We then show the tight connection that there is
between algebras of the extended probabilistic powerdomain monad and
barycentres in a sense inspired from Choquet \cite{Choquet69}, and
already used by Cohen, Escard\'o and Keimel in \cite{cohen06}:
the 
structure maps of algebras map every continuous valuation to one of
its barycentres, and conversely, if barycentres are unique and the
barycentre map is continuous, then it is the structure map of an
algebra.  Moreover, on so-called \emph{locally linear} cones, the mere
existence of barycentres defines an algebra, and on convex-$T_0$
cones, all continuous linear maps are algebra morphisms.  Compared to
\cite{cohen06}, we do not need any stable compactness assumption, and
this is due to the Schr\"oder-Simpson theorem (see
Section~\ref{sec:riesz}).  We also isolate the new notion of local
linearity, which seems to have been overlooked in ibid.

In Section~\ref{fandp}, we consider two related probabilistic
powerdomain constructions, the \emph{simple valuation monad}~$\VF$ and
the \emph{point-continuous valuation monad}~$\VP$. Those were
initially considered by Heckmann in~\cite{heckmann96}. We fully
characterise the algebras of this two monads as weakly locally convex
topological cones and weakly locally convex sober topological cones,
respectively. In both cases, the algebra morphisms are shown to be
continuous linear maps between the corresponding algebras.  Those are
simple consequences of Heckmann's results.


\subsection*{Preliminaries.}

We use standard notions and notations in domain
theory~\cite{gierz03,abramsky94} and in non-Hausdorff
topology~\cite{goubault13a}. The category of topological spaces and
continuous functions is denoted by {\bf TOP}. For the convenience of
our discussion, we restrict ourselves to its full subcategory $\ctop$
of $T_0$ topological spaces. The category of dcpos and
Scott-continuous functions is denoted by~{\bf DCPO}. We use
$\mathbb R_+$ to denote the set of positive reals, and $\exr$ to
denote the positive reals extended with $\infty$. The extended
positive reals $\exr$ will play a vital role in our discussion.
Whenever $\exr$ is treated as a topological space, we mean that it is
equipped with the Scott topology until stated otherwise.

\section{The extended probabilistic powerdomain monad}
\label{sec:V}

\subsection{The extended probabilistic powerdomain functor}

\begin{definition}
  A \emph{valuation} on a topological space $(X,\mathcal OX)$ is a
  function~$\mu$ from $\mathcal OX$ to the extended positive reals
  $\exr$ satisfying for any $U, V\in \mathcal OX$:
  \begin{itemize}
  \item (strictness)  $\mu (\emptyset) = 0$;   
  \item (monotoncity) $\mu(U)\leq \mu(V)$ if $U\subseteq V$; 
  \item (modularity)
    $ \mu(U) + \mu(V) = \mu (U\cup V) +\mu (U\cap V)$.
  \end{itemize}

  A \emph{continuous valuation} $\mu$ on $(X, \mathcal OX)$ is a
  valuation that is Scott-continuous from $\mathcal OX$ to $\exr$,
  that is, for every directed family of open subsets $U_i, i\in I$, it
  holds:
  \begin{itemize}
  \item (Scott-continuity)
    $\mu (\bigcup_{i\in I}U_i) = \sup_{i\in I} \mu(U_i)$.
  \end{itemize}

  Valuations on the same topological space~$X$ are ordered by
  $\mu\leq \nu$ if and only if $\mu(U)\leq \nu(U)$ for all
  $U\in \mathcal OX$. The order is sometimes referred as the \emph
  {stochastic order}.

  The set of continuous valuations on $X$ with the stochastic order is
  denoted by~$\mathcal VX$.
\end{definition}

\begin{example}
  Let $X$ be a topological space, the \emph{Dirac mass} $\delta_x$ at
  $x\in X$ is defined by $\delta_x(U)=1$ if $x\in U$ and~$0$
  otherwise.  The Dirac mass $\delta_x$ is a continuous valuation
  on~$X$ for every $x\in X$.
\end{example}

\begin{example}\label{finitevaluations}
  Let $X$ be a topological space, The linear combinations
  $\sum_{i=1}^n r_i\delta_{x_i}$ of Dirac masses are also continuous
  valuations, where $r_i\in \mathbb R_+, x_i\in X$.  These valuations
  are called \emph{simple valuations}. The set of all simple
  valuations on~$X$ is denoted as $\VF X$.
\end{example}

\begin{example}
  Let $X$ be a topological space, $\mu, \nu$ be continuous valuations
  on~$X$ and $r, s\in \mathbb R_+$. The linear combinations
  $r\mu+s\nu$, defined as
  $(r\mu+s\nu)(U)= r\cdot \mu(U)+s\cdot \nu(U)$ for every open subset
  $U$, are again continuous valuations.
\end{example}

\begin{proposition}
  $\mathcal VX$ is a dcpo in the stochastic order.
\end{proposition}
\begin{proof}
  For a directed family of continuous valuations $\mu_i, i\in I$, and
  any open subset $U\subseteq X$, define
  $(\sup_{i\in I}\mu_i) (U)=\sup_{i\in I} \mu_i(U)$. One verifies that
  $\sup_{i\in I}\mu_i$ is another continuous valuation. See
  \cite[Lemma IV-9.8]{gierz03} and \cite[Section 3.2.(5)]{heckmann96}
  for details.
\end{proof}

We can extend $\mathcal V$ to a functor from the category of
topological spaces and continuous functions to the category of dcpos
and Scott-continuous functions by using the following proposition.

\begin{proposition}
  Let $f\colon  X\to Y$ be a continuous function between topological
  spaces~$X$ and~$Y$, and $\mu$ be any continuous valuation
  on~$X$. Then the map
  $\mathcal V f\colon \mu\mapsto (U\in \mathcal OY\mapsto \mu(f^{-1}(U)))$
  is Scott-continuous from $\mathcal VX$ to~$\mathcal V Y$.
\end{proposition}
\begin{proof}
  Straightforward. 
\end{proof}

\begin{corollary}
  $\mathcal V$ is a functor from the category {\bf TOP} to the
  category~{\bf DCPO}.
\end{corollary}
\begin{proof}
  Straightforward. 
\end{proof}

There is a canonical functor $\Sigma$ from the category {\bf DCPO} to
${\bf TOP}$, namely, the Scott-space construction. For any dcpo $L$,
$\Sigma L$ is the topological space $(L, \sigma L)$, where $\sigma L$
is the Scott topology on~$L$; and for any Scott-continuous function
$f\colon  L\to M $, $\Sigma f=f$ is continuous from $\Sigma L$ to
$\Sigma M$.

Post-composing the functor $\Sigma$ with $\mathcal V$, one obtains an
endofunctor $\Sigma\circ \mathcal V$ over the category~{\bf TOP}, we
denote it by $\mathcal V_s$. Pre-composing the functor $\Sigma$ with
$\mathcal V$, however, yields an endofunctor $\mathcal V\circ \Sigma$
over the category~{\bf DCPO}, we denote it by $\mathcal V_d$.

In her 
PhD thesis~\cite{jones90}, Jones showed that $\mathcal V_d$ is a monad
over the category~{\bf DCPO}, and moreover, $\mathcal V_d$ can be
restricted to the full subcategory of continuous domains. She then
used this monad to model probabilistic side effects of programming
languages.

Naturally, one wonders whether $\mathcal V_s$ is a monad on~{\bf
  TOP}.  Unfortunately, this is not the case in general. The problem is
that the topology $\mathcal OX$ is, in general, too sparse to
sufficiently restrict the Scott topology on $\mathcal
VX$. Alternatively, one considers the \emph{weak topology}
on~$\mathcal VX$, and we will see this is the right topology on
$\mathcal VX$ that gives rise to a monad structure.

\begin{definition} {\rm\cite[Satz 8.5]{kirch93}}
  For a topological space $X$, the \emph{weak topology} on
  $\mathcal VX$ is generated by a subbasis of sets of the form
  $[U>r], U\in \mathcal OX, r\in \mathbb R_+$, where $[U>r]$ denotes
  the set of continuous valuations $\mu$ such that $\mu(U)>r$. We use
  $\mathcal V_wX$ to denote the space $\mathcal VX$ equipped with the
  weak topology and call $\V X$ the \emph{extended probabilistic
    powerdomain} or the \emph{valuation powerdomain} over~$X$.
\end{definition}

Analogously, we can extend $\mathcal V_w$ into a functor
on~${\bf TOP}$ by defining its actions $\mathcal V_w f$ on continuous
maps $f \colon X\to Y$ by $\mathcal V_wf(\mu)(V)=\mu(f^{-1}(V))$.

\begin{proposition}
  \label{prop:V:functor}
  $\mathcal V_w$ is an endofunctor on the category~{\bf TOP}.
\end{proposition}
\begin{proof}
  The main thing is to check that $\V f$ is continuous for every
  continuous map $f \colon X \to Y$.  For every open subset $V$ of
  $Y$, for every $r \in \Rplus \setminus \{0\}$,
  $(\V f)^{-1} ([V > r]) = \{\mu \in \mathcal VX \mid \mu (f^{-1} (V))
  > r\} = [f^{-1} (V) > r]$.
\end{proof}

\subsection{Integral with respect to continuous valuations}

Continuous valuations are variations on the idea of measure. While
measures allow one to integrate measurable functions, continuous
valuations allow one to integrate lower semi-continuous functions.  A
\emph{lower semi-continuous function} from a topological space to
$\exr$ is the same thing as a continuous function from $X$ to $\exr$,
where the latter is equipped with the Scott topology.  We
write $\mathcal LX$ for the set of lower semi-continuous functions from
$X$ to $\exr$.

For any topological space~$X$, every lower semi-continuous function
$h\colon X\to \exr$ has a Choquet type \emph{integral} with respect to
a continuous valuation~$\mu$ on~$X$ defined by:
$$\int_{x\in X} h(x)d\mu = \int_{0}^\infty \mu (h^{-1}(r, \infty])dr,$$
where the right side of the equation is a Riemann integral.  If no
risk of confusion occurs, we usually write $\int_{x\in X} h(x)d\mu$ as
$\int h~d\mu$. For the discussion that follows, we collect some
properties of this integral, and readers are referred
to~\cite{kirch93,tix95,lawson04} for more details.

\begin{lemma} \label{propertyofint}
  \begin{enumerate}
  \item For every simple valuation
    $\mu = \sum_{i=1}^n r_i\delta_{x_i}$,
    $\int h~d\mu = \sum_{i=1}^n h(x_i)$. In particular, for the Dirac
    mass $\delta_x$, $\int h~d \delta_x=h(x)$.
  \item For all lower semi-continuous functions $h, k \colon X\to \exr$ and
    $r, s\in \mathbb R_+$,
    $\int (r h+s k)~d\mu= r\int h~d\mu +s\int k~d\mu$.
  \item For every directed family (in the pointwise order) of lower
    semi-continuous functions $h_a \colon X \to \exr, a\in A$, we have
    that $\int (\sup_{a\in A}h_a)~d\mu= \sup_{a\in A}\int h_a~d\mu$.
  \item For every open set~$U$, $\int \chi_U~d\mu = \mu(U)$, here
    $\chi_U$ is the characteristic function of~$U$ defined as
    $\chi_U(x)=1$ when $x\in U$ and $0$ otherwise.
  \item For all continuous valuations $\mu, \nu\in \mathcal V_wX$ and
    $r, s\in \mathbb R_+$, for every lower semi-continuous function
    $f \colon X \to \exr$,
    $\int f~d(r\mu+s\nu) = r\int f~d\mu + s\int f~d\nu$.
  \item Let $f\colon  X\to Y$ be a continuous map, $\mu$ be a continuous
    valuation on~$X$, and $g\colon Y \to \exr$ be a lower semi-continuous
    function. Then
    $\int_{y\in Y} g(y)d\mathcal V_wf(\mu) = \int_{x\in X} (g\circ
    f)(x) d\mu$.
  \end{enumerate}
\end{lemma}

Those properties imply a form of the Riesz representation theorem for
continuous valuations~\cite{kirch93}. It states that integrating with
respect to a continuous valuation $\nu$ defines a lower semi-continuous
linear functional $f\mapsto \int f~d\nu$ on $\mathcal LX$ and that,
conversely, for every lower semi-continuous linear functional $\phi$ on
$\mathcal LX$, there is a unique continuous valuation $\nu$
representing $\phi$, in the sense that $\phi(f) = \int f~d\nu$ for
every $f \in \mathcal LX$, and $\nu$ is given by
$\nu (U) = \phi (\chi_U)$ for every open set $U$.

For all $h\in \mathcal LX$ and $r\in \mathbb R_+$, we define
$[h>r]=\set{\mu\in \mathcal V_wX}{\int h~d\mu >r}$. It is routine to
check that $[h> r]$ are open in the weak topology of $\mathcal V_w
X$. They also form a subbase of the weak topology, as
$[U>r]=[\chi_U>r]$.

\subsection{The monad structure}

Using integration, we now argue that $\mathcal V_w$ defines a monad on
the category~{\bf TOP}. Recall that a monad on a category~{\bf C}
consists of an endofunctor~$T\colon  {\bf C}\to {\bf C}$ together with two
natural transformations: $\eta\colon  1_C\to T$ (where $1_{C}$ denotes the
identity functor on {\bf C}) and $m\colon  T^2\to T$, satisfying the
equalities $m \circ Tm =m \circ m T$ and
$m \circ T\eta =m \circ \eta T=1_{T}$. The natural transformations
$\eta$ and $m$ are called the \emph{unit} and the
\emph{multiplication} of the monad, respectively.  Alternatively, one
can use the following equivalent description, due to Manes.

\begin{definition}{\rm \cite{manes76}}
  A \emph{monad} on a category {\bf C} is a triple
  $(T , \eta, \_^\dagger)$ consisting of a map $T$ from objects $X$
  of~{\bf C} to objects $T X$ of {\bf C}, a collection
  $\eta = (\eta_X)_X$ of morphisms $\eta_X \colon  X \to TX$, one for each
  object $X$ of {\bf C}, and a so-called extension operation
  $\_^\dagger$ that maps every morphism $ f \colon  X\to TY$ to
  $f ^\dagger \colon  T X\to T Y$ such that:
  \begin{enumerate}
  \item $\eta_X^\dagger = \id_{TX}$;
  \item for every morphism $f \colon X\to TY$,
    $ f^\dagger\circ \eta_X = f$;
  \item for all morphisms $f \colon X\to TY$ and $g \colon Y\to TZ$,
    $g^\dagger\circ f^\dagger = (g^\dagger\circ f)^\dagger$.
\end{enumerate}
\end{definition}

The advantage of this definition is that one does not even need to
verify that $T$ is a functor before checking it is a monad. In fact
every monad defined in this sense gives rise to an endofunctor, by
defining its action on morphisms $f \colon X\to Y$ as
$Tf = (\eta_Y\circ f)^\dagger$. The unit of the monad $\eta$ is given
by $(\eta_X)_X$ and the multiplication~$m$ is given by
$m_X= \id_{TX}^\dagger$ for every object $X$ in~{\bf C}.

The following is folklore, and is implicit in \cite{cohen06}, for
example. 

\begin{proposition}\label{folklore}
  The functor $\mathcal V_w$ is a monad on the category~{\bf TOP}. The
  unit $\eta$ is given by $\eta_X \colon x\mapsto \delta_x$ for
  every~$X$, and for continuous function
  $f \colon X\to \mathcal V_wY $ the extension operation is given by
  $$f^\dagger (\mu)(U) = \int_{x\in X} f(x)(U)d\mu.$$
  For every lower semi-continuous function $h \colon Y \to \exr$, the
  following disintegration formula holds:
  \begin{align}
    \label{eq:disint}
    \int_{y\in Y}h(y)d f^\dagger(\mu)
    & = \int_{x\in X}\left(\int_{y\in Y}h(y)df(x) \right)d\mu.
  \end{align}
  In particular, the function $x\mapsto \int_{y\in Y}h(y)df(x)$ is
  lower semi-continuous.
\end{proposition}
\begin{proof}
  The map $x \in X \mapsto f (x) (U)$ is continuous for every open set
  $U$ by the definition of the weak topology, hence the formula makes
  sense.  We directly prove the last claim.  Let us assume that
  $x \in X$ satisfies $\int_{y\in Y}h(y)df(x) > r$, where
  $r \in \Rplus \setminus \{0\}$.  The function $h$ is the supremum of
  the countable chain of maps $h_N$, defined as
  $\frac 1 {2^N} \sum_{k=1}^{N2^N} \chi_{h^{-1} (k/2^N, \infty]}$, so
  $\int_{y \in Y} h_N (y) df (x) > r$ for some $N \in \nat$.  Let us
  write $h_N$ as $\epsilon \sum_{k=1}^n \chi_{U_k}$ (a so-called
  \emph{step function}), where $\epsilon >0$ and each $U_k$ is open,
  to avoid irrelevant details.  Then
  $\epsilon \sum_{k=1}^n f (x) (U_k) > r$, so there are numbers
  $r_k \in \Rplus \setminus \{0\}$ such that $f (x) (U_k) > r_k$ for
  each $k$ and $\epsilon \sum_{k=1}^n r_k \geq r$.  Then
  $\bigcap_{k=1}^n f^{-1} ([U_k > r_k])$ is an open neighbourhood of
  $x$, and
  $\int_{y \in Y} h (y) d f (x') \geq \int_{y \in Y} h_N (y) d f (x')
  = \epsilon \sum_{k=1}^n f (x') (U_k) > r$ for every $x'$ in that
  neighbourhood.

  Let us define $\Lambda (h)$ as
  $\int_{x\in X}\left(\int_{y\in Y}h(y)df(x) \right)d\mu$ for every
  $h \in \Lform X$, which now makes sense.  It is easy to see that
  $\Lambda$ is linear and lower semi-continuous, hence there is a
  unique continuous valuation $\nu$ such that
  $\Lambda (h) = \int h~d\nu$ for every $h \in \Lform Y$.  We have
  $\nu (U) = \Lambda (\chi_U)$, and this gives us back the definition
  of $f^\dagger (\mu) (U)$.

  It remains to check the monad equations.  That could be done as in
  \cite{kirch93}, but Manes' formulation makes it easier.  Equations
  (i) and (ii) are immediate.  For (iii), we have:
  \begin{align*}
    (g^\dagger\circ f^\dagger) (\mu) (U)
    & = \int_{y \in Y} g (y) (U) df^\dagger (\mu) \\
    & = \int_{x \in X} \left(\int_{y \in Y} g (y) (U) df (x)\right) d\mu
    & \text{by (\ref{eq:disint})},
  \end{align*}
  while:
  \begin{align*}
    (g^\dagger \circ f)^\dagger (\mu) (U)
    & = \int_{x \in X} g^\dagger (f (x)) (U) d\mu \\
    & = \int_{x \in X} \left(\int_{y \in Y} g (y) (U) df (x)\right) d\mu
    & \text{by definition.}
  \end{align*}
\end{proof}

\begin{remark}
  For a topological space $X$, the multiplication $m_X$ of the monad
  at $\mathcal V_wX$ sends every continuous valuation
  $ \varpi \in \mathcal V_w(\mathcal V_wX)$ to
  $\id_{\V X}^\dagger (\varpi)= (U\mapsto \int_{\mu\in \V
    X}\mu(U)d\varpi)$. In particular, for any continuous valuation
  $\mu\in \mathcal V_wX$, $m_X(\delta_\mu)=\mu$.
\end{remark}


In this paper, we are mainly interested in the Eilenberg-Moore
algebras of the valuation powerdomain monad over the
category~$\ctop$. Recall that an \emph{algebra of a monad}
$\mathcal T$ over category {\bf C} is a pair $(A, \alpha)$, where $A$
is an object in {\bf C} and $\alpha_A\colon \mathcal TA\to A$ is a
morphism of~{\bf C}, called the \emph{structure map}, such that
$\alpha_A \circ \eta_A=\id_A$ and
$\alpha_A\circ m_A=\alpha_A\circ \mathcal T\alpha_A$. A morphism
$f\colon A\to B$ in~{\bf C} is called a \emph{$\mathcal T$-algebra
  morphism} if $f\circ \alpha_A=\mathcal Tf\circ \alpha_B$.

From the basic theory of algebras of monads, we know that in
particular the pair $(\mathcal T A, m_A)$ is an algebra of~$T$, where
$m$ is the multiplication of $\mathcal T$. In our case,
$(\mathcal V_w X, m_X)$ is a $\mathcal V_w$-algebra for every
topological space~$X$. In order to locate all the algebras, let us
first examine the structure of $\mathcal V_wX$ for an arbitrary
topological space~$X$. We will see that $\mathcal V_wX$ is a
\emph{topological cone} as defined below.


\section{Cones}\label{cones}

\subsection{Topological, locally convex, and locally linear cones}

The following notions are from \cite{Keimel:topcones2}.
\begin{definition}
  A \emph{cone} is a commutative monoid~$C$ together with a scalar
  multiplication by nonnegative real numbers satisfying the same
  axioms as for vector spaces; that is, $C$ is endowed with an
  addition $(x, y)\mapsto x+y \colon C\times C\to C$ which is
  associative, commutative and admits a neutral element~$0$, and with
  a scalar multiplication
  $(r, x)\mapsto r \cdot x \colon \mathbb R_+ \times C\to C$
  satisfying the following axioms for all $x, y\in C$ and all
  $r, s\in \exr$:
  \begin{align*}
    &r\cdot (x+y) = r\cdot x+r\cdot y         & &  (rs)\cdot x=r\cdot (s\cdot x)        &  0\cdot x = 0 \\
    & (r+s)\cdot x = r\cdot x+s\cdot x       & &1\cdot x=x                                     &r\cdot 0 = 0
  \end{align*}

  We shall often write $rx$ instead of $r \cdot x$ for
  $r\in \mathbb R_+$ and $x\in C$.

  A \emph{semitopological cone} is a cone with a $T_0$ topology that
  makes $+$ and $\cdot$ separately continuous.

  A \emph{topological cone} is a cone with a $T_0$ topology that makes
  $+$ and $\cdot$ jointly continuous.
\end{definition}

\begin{remark}
  \label{rem:ershov}
  If $\cdot$ is separately continuous, it is automatically jointly
  continuous \cite[Corollary~6.9~(a)]{Keimel:topcones2}.  This is a
  consequence of a theorem due to Ershov
  \cite[Proposition~2]{Ershov:aspace:hull}, which states that every
  separately continuous map from $X \times Y$ to $Z$ where $X$ is a
  c-space (in particular, a continuous poset in its Scott topology)
  and $Y$ and $Z$ are arbitrary spaces is jointly continuous.
\end{remark}

\begin{definition}
  A function $f \colon C\to D$ from cone $C$ to $D$ is \emph{linear}
  if and only if for all $r, s\in \mathbb R_+$ and $x, y\in C$,
  $f(rx+sy)= rf(x)+sf(y)$.
\end{definition}

\begin{example}
  The extended reals $\exr$ is a topological cone in the Scott
  topology, with the usual addition and multiplication extended with
  $r+\infty= \infty+r =\infty$ for all $r\in \exr$,
  $0\cdot\infty = \infty\cdot 0=0$, and
  $r\cdot \infty=\infty\cdot r=\infty$ for $r\not= 0$.
\end{example}

\begin{example}
  \label{exa:LX:top}
  For any topological space~$X$, $\mathcal LX$ is a cone with the
  pointwise addition and multiplication. It is a semitopological cone
  with the Scott topology induced by the pointwise order.  It is a
  topological cone if $X$ is core-compact (i.e., if $\Open X$ is a
  continuous lattice).  Indeed, in that case $\mathcal LX$ is also a
  continuous lattice; this can be obtained from
  \cite[Proposition~II-4.6]{gierz03} and the fact that $\creal$ is a
  continuous lattice.  Every continuous dcpo is a c-space in its Scott
  topology, then we use \cite[Corollary~6.9~(c)]{Keimel:topcones2},
  which says that every semitopological cone with a c-space topology
  is topological.
%
\end{example}

\begin{example}
  \label{exa:L}
  For every bounded sup-semi-lattice $(L, \leq, \top, \vee)$, we can
  define $x+y$ as $x \vee y$, $r \cdot x$ as $x$ if $r > 0$, $\bot$
  otherwise.  This is a cone.  With the Scott topology, it is a
  semitopological cone, and a topological cone if $L$ is continuous
  \cite[Section~6.1]{heckmann96}.  This illustrates that how far from
  vector spaces cones can be.
\end{example}

\begin{example}\label{dualcones}
  \begin{enumerate}
  \item For any cone~$C$, the set of all linear maps from $C$ to
    $\exr$ is a cone with pointwise addition and scalar
    multiplication.
  \item For any semitopological cone~$C$, the set of all \emph{lower
      semi-continuous} linear maps from $C$ to $\exr$ is a cone with
    pointwise addition and multiplication. We denote it as $C^*$ and
    call it the \emph{dual cone} of~$C$. We endow $C^*$ with the
    \emph{upper weak$^*$ topology}, that is, the coarsest topology
    making the functions
    $$\eta_C(x) = (\phi \mapsto \phi(x)) \colon C^*\to \exr$$
    continuous for all $x\in C$.  The cone $C^*$ with the upper
    weak$^*$ topology is a topological cone, as is every subcone of
    any power ${\exr}^I$ with the subspace topology of the product
    topology, see the discussion after Definition~5.1 in
    \cite{Keimel:topcones2}, or \cite[Section~3]{cohen06} for example.
  \end{enumerate}
\end{example}

\begin{proposition}
  For any topological space $X$, $\mathcal V_wX$ is a $T_0$
  topological cone.
\end{proposition}
$\mathcal V_wX$ can be identified with the dual cone
${(\mathcal LX)}^*$, by a form of the Riesz representation theorem
~\cite{kirch93}; see also Section~\ref{sec:riesz}.  This is the path
taken in \cite{cohen06}.  We give an explicit proof.  Showing that
$C^*$, for a general semitopological cone $C$, is a $T_0$ topological
cone is done similarly. 

\begin{proof}
  For all continuous valuations $\mu, \nu \in \mathcal V_wX$ and
  $r\in \Rplus$, we define $(r\cdot\mu) (U) = r(\mu (U))$ and
  $(\mu+\nu)(U)= \mu(U)+\nu(U)$. It is easy to see that
  $\mathcal V_wX$ with $+$ and $\cdot$ form a cone structure. We
  proceed to check that $+$ and~$\cdot$ are jointly continuous. To
  this end, we assume that $\mu + \nu\in [U>r]$ for some $U$ open
  in~$X$ and $r\in \mathbb R_+\setminus\{0\}$. By definition, that
  means that $\mu (U)+\nu(U)>r$. If either $\mu(U)$ or $\nu (U)$ is
  equal to $\infty$, say $\mu(U)= \infty$, then we know that
  $\mu\in [U>r]$, and we pick the whole $\mathcal V_wX$ as an open
  neighbourhood of $\nu$. Obviously, for any $\mu'\in [U>r]$ and
  $\nu'\in \mathcal V_wX$, $\mu'+\nu'\in [U>r]$. If $\mu (U)+\nu(U)$
  is finite, we choose some $s\in \mathbb R_+$ such that
  $\mu (U)+\nu(U)>s>r$, we let $ \varepsilon = \frac{s-r}{2}$,
  $r_\mu = \max\{\mu(U)-\varepsilon, 0\}$ and
  $r_\nu = \max\{\nu(U)-\varepsilon, 0\}$. Then $\mu \in [U>r_\mu]$
  and $\nu\in [U>r_\nu]$, and for all $\mu'\in [U>r_\mu]$ and
  $\nu'\in [U>r_\nu]$,
  $(\mu'+\nu')(U)= \mu'(U) + \nu'(U)>r_\mu+r_\nu>r$. So we have proved
  that $+$ is jointly continuous.  The joint continuity of scalar
  multiplication can be verified similarly.

  For $T_0$-ness, let $\mu_1$ and $\mu_2$ be two different continuous
  valuations. Then there exists an open set $U$ such that
  $\mu_1(U)\not=\mu_2(U)$. Without loss of generality we assume that
  $\mu_1(U)<\mu_2(U)$. Choose $s$ such that
  $\mu_1(U)<s<\mu_2(U)$. Then $[U>s]$ is an open subset of
  $\mathcal V_wX$ containing $\mu_2$ but not $\mu_1$.
\end{proof}

The cone structure on~$\mathcal V_wX$ also has additional properties.
\begin{definition}
  \begin{itemize}
  \item A subset $A$ of a cone $C$ is called \emph{convex} if and only
    if, for any two points $a, b\in A$, the linear combination
    $ra+(1-r)b$ is in~$A$ for any $r\in [0,1 ]$.
  \item A subset $A$ of a cone C is called a \emph{half-space} if and
    only if both $A$ and its complement are convex.
  \item A cone $C$ with a $T_0$ topology is called \emph{weakly
      locally convex} if and only if for every point $x\in C$, every
    open neighbourhood $U$ of $x$ contains a convex (not necessarily
    open) neighbourhood of $x$.
  \item A cone $C$ with a $T_0$ topology is called \emph{locally
      convex} if and only if each point has a neighbourhood basis of
    open convex neighbourhoods.
  \item A cone $C$ with a $T_0$ topology is called \emph{locally
      linear} if and only if $C$ has a subbase of open half-spaces.
  \end{itemize}
\end{definition}
Weak local convexity was introduced in \cite{heckmann96}, where it is
simply called local convexity.  Our notion of local convexity is that
of \cite{Keimel:topcones2,cohen06}.  The notion of local linearity is
new.  Note that all those notions would be equivalent in the context
of topological vector spaces.

\begin{proposition}
  Every locally linear topological cone is locally convex, and every
  locally convex topological cone is weakly locally convex. \qed
\end{proposition}

\begin{example}
  The dual cone $C^*$ of any semitopological cone~$C$ (defined in
  Example~\ref{dualcones}) is locally linear. One verifies that the
  sets $(\eta_C(x))^{-1}((r, \infty])$ are half-spaces for all
  $x\in X$ and $r\in \mathbb R_+$, and they form a subbase for the
  upper weak$^*$ topology on~$C^*$.
\end{example}

\begin{example}
  Specializing the previous construction to
  $\mathcal V_wX \cong {(\mathcal LX)}^*$, the subbasic open subsets
  $[U > r]$ of $\mathcal V_wX$ are all half-spaces, so $\mathcal V_wX$
  is locally linear.
\end{example}

The topology on $\mathcal V_wX$ is more than $T_0$, it is actually
sober by~\cite[Proposition 5.1]{heckmann96}.  Hence:
\begin{proposition}
  $\mathcal V_wX$ is a locally linear sober topological cone for any
  space~$X$.  \qed
\end{proposition}

\begin{example}
  \label{exa:LX:locconv}
  For every core-compact space $X$, $\Lform X$ is a continuous
  lattice.  It follows that $\Lform X$ (with its Scott topology) is a
  locally convex topological cone, using an argument that Keimel
  \cite[Lemma~6.12]{Keimel:topcones2} attributes to Lawson.
  We argue that $\Lform X$ is in fact locally linear.  More generally,
  $\Lform X$ is a locally linear semitopological cone for every space
  $X$ whose sobrification $X^s$ is $\odot$-consonant
  \cite[Definition~13.1]{dBGLJL:LCS}.  (If $X$ is core-compact, then
  $X^s$ is locally compact sober \cite[Theorem~8.3.10]{goubault13a},
  every locally compact sober space is LCS-complete, and every
  LCS-complete space is $\odot$-consonant
  \cite[Lemma~13.2]{dBGLJL:LCS}.)  First, $\Lform X$ is homeomorphic
  to $\Lform {X^s}$, where $X^s$ is the sobrification of $X$
  \cite[Lemma~2.1]{JGL-mscs16}.  This is because $\creal$ is sober,
  and therefore every continuous map from $X$ to $\creal$ has a unique
  continuous extension to $X^s$.  This homeomorphism is also an
  isomorphism of cones.  If $X^s$ is $\odot$-consonant, then the Scott
  topology on $\Lform {X^s}$ coincides with the compact-open topology
  \cite[Corollary~13.5]{dBGLJL:LCS}.  The subbasic open subsets
  $\{f \in \Lform {X^s} \mid f (Q) \subseteq (r, \infty]\}$ ($Q$
  compact saturated in $X^s$, $r \in \Rplus$) are easily seen to be
  open half-spaces.
\end{example}

\begin{example}
  \label{exa:L:locconv}
  Here is an example of a locally convex, non-locally linear 
  topological cone.  Consider any complete lattice $L$, and equip it
  with the Scott topology and with the cone structure of
  Example~\ref{exa:L}.  Its non-empty convex subsets are its directed
  subsets.  In particular, every open subset is convex, which implies
  that $L$ is trivially locally convex.  For every non-empty convex
  closed subset $C$, $C$ is directed and closed, so $x=\sup C$ is in
  $C$, and therefore $C$ is the downward closure $\dc x$ of $x$.
  Hence the proper open half-spaces are exactly the complements of
  downward closures of points.  It follows that the topology generated
  by the open half-spaces is the \emph{upper topology}.  In
  particular, $L$ is locally linear if and only if the upper and Scott
  topologies coincide.  In particular, for a continuous (complete)
  lattice $L$, $L$ is locally linear if and only if $L$ is
  hypercontinuous \cite[Proposition~VII-3.4]{gierz03}.  The
  distributive hypercontinuous lattices are the Stone duals of
  quasi-continuous dcpos \cite[Propositions~VII-3.7,
  VII-3.8]{gierz03}.  Hence any lattice of the form $\Open X$, where
  $X$ is core-compact but not a quasi-continuous dcpo (or does not
  have the Scott topology), is a locally convex topological cone that
  is not locally linear.  For example $\Open\real$, where $\real$ 
  comes with its metric topology, fits.
\end{example}

\begin{remark}
  We have already mentioned that local linearity was not used in
  \cite{cohen06}, and one may think that this is due to the author's
  reliance on stable compactness.  However, there are stably compact,
  locally convex but non-locally linear topological cones: any
  continuous, non-hypercontinuous 
  lattice $L$ will serve as an example (Example~\ref{exa:L:locconv}),
  since every continuous lattice is stably compact in its Scott
  topology \cite[Fact~9.1.6]{goubault13a}.
\end{remark}

Recall that a \emph{retraction} $r\colon X\to Y$ is a continuous map between
topological spaces such that there is a continuous map $s\colon  Y\to X$
with $ r\circ s = \id_Y$. $Y$ is the retract of $X$. A \emph{linear
  retraction} is any retraction $r\colon C\to D$ between semitopological
cones that is also linear. Then $D$ is a \emph{linear retract} of
$C$. Beware that we do \emph{not} require the associated
\emph{section}~$s$ to be linear in any way.

Heckmann showed that every linear retract of a weakly locally convex
cone is weakly locally convex \cite[Proposition 6.6]{heckmann96}.  It
follows:
\begin{proposition}\label{suls}
  Let $C$ be a locally linear topological cone, $D$ be a topological
  cone, and $r \colon C\to D$ be a linear retraction. Then $D$ is a
  weakly locally convex cone.  \qed
\end{proposition}

We will see in Section~\ref{fandp} that, conversely, every weakly
locally convex cone is a linear retract of some locally linear
topological cone.

\emph{Keimel's Separation Theorem}, which we reproduce below, is an
analogue of the Hahn-Banach separation theorem on 
semitopological cones, and provides us with a rich collection of lower
semi-continuous linear maps.
\begin{theorem}{\rm \cite[Theorem
    9.1]{Keimel:topcones2}} \label{keimelsep} In a semitopological
  cone $C$ consider a nonempty convex subset~$A$ and an open convex
  subset $U$. If $A$ and $U$ are disjoint, then there exists a lower
  semi-continuous linear functional $\Lambda \colon C\to \exr$ such
  that $\Lambda(x) \leq 1<\Lambda(y)$ for all $x\in A$ and $y\in
  U$. \qed
\end{theorem}

Following Keimel, we call a semitopological cone $C$
\emph{convex-$T_0$} if and only if for every pair of distinct points
$a$, $b$ of $C$, there is a lower semi-continuous linear function
$\Lambda \colon C\to \exr$ such that $\Lambda(a)\not= \Lambda(b)$
\cite[Definition~4.7]{Keimel:topcones2}.  The following is an
immediate consequence of \cite[Corollary~9.3]{Keimel:topcones2}.  We
give the explicit, short proof.
\begin{corollary}\label{coroab}
  Every locally convex semitopological cone is convex-$T_0$.
\end{corollary}
\begin{proof}
  Since $C$ is $T_0$, we may assume that there exists an open open $U$
  containing $a$ but not $b$. Since $C$ is locally convex, we can find
  an open convex subset $V$ such that $a\in V\subseteq U$.  Realising
  that the singleton set $\{b\}$ is a convex set and $b\notin V$, we
  apply Theorem~\ref{keimelsep} and we find a lower semi-continuous
  linear functional $\Lambda$ such that $\Lambda (b)\leq 1<\Lambda(y)$
  for all $y\in V$. Hence $\Lambda(b)<\Lambda(a)$, since $a\in V$.
\end{proof}

Linear maps on cones such as $\creal$, $\Lform X$, $\mathcal V_wX$
follow our intuition.  Let us explore the stranger cones from
Example~\ref{exa:L}.
\begin{example}
  \label{exa:L:lin}
  Consider any complete lattice $L$ with its Scott topology and the
  cone structure of Example~\ref{exa:L}.  For every lower
  semi-continuous linear map $\Lambda \colon L \to \creal$,
  $\Lambda^{-1} ((1, \infty])$ is a proper open half-space, hence of
  the form $L \setminus \dc x_0$ for some point $x_0 \in L$ (see
  Example~\ref{exa:L:locconv}).  Then $x \leq x_0$ if and only if
  $\Lambda (x) < 1$ for every $x \in L$, and the equality
  $\Lambda (rx) = r \Lambda (x)$ implies that $\Lambda (x)$ can only
  be equal to $0$ or to $\infty$.  It follows that the semi-continuous
  linear maps on $L$ are exactly the maps
  $\infty \cdot \chi_{L \setminus \dc x_0}$, where $x_0 \in X$.
  
  As a consequence, the dual cone $L^*$ can be equated with the
  opposite lattice $L^{op}$ with the upper topology.  The cone
  structure is that of $L^{op}$: addition is infimum in $L$, $r \cdot
  x$ is equal to $x$ if $r \neq 0$, to the top element of $L$ otherwise.
\end{example}

\subsection{A Riesz-type representation and the Schr\"oder-Simpson
  Theorem}
\label{sec:riesz}

We have already mentioned a Riesz-type representation theorem for
continuous valuations~\cite{kirch93}.  That states that
$\nu \mapsto (f \mapsto \int f~d\nu)$ and $\phi \mapsto (U \mapsto \phi (\chi_U))$ define mutually inverse maps between continuous valuations on $X$
and lower semi-continuous linear functions on $\mathcal LX$.  Additionally,
those define a homeomorphism between $\mathcal V_w X$ and the dual cone
$(\mathcal LX)^*$, namely, the weak topology on the former is in one-to-one
correspondence with the upper weak$^*$ topology on the latter under this bijection.

There is yet another representation theorem, the so-called
Schr\"oder-Simpson Theorem, stating that any linear lower
semi-continuous functional $\phi$ from~$\mathcal V_wX$ to~$\exr$ is
uniquely determined by a semi-continuous function $h\in \mathcal LX$
in the sense that $\phi(\nu) = \int h~d\nu$ for all
$\nu\in \mathcal V_wX$. The theorem was originally proved by
Schr\"oder and Simpson~\cite{schroder05}, Keimel gave a conceptual
proof of it in~\cite{keimel12}, and the first author gave an
elementary proof in~\cite{goubault15}.

\begin{theorem}{\rm (The Schr\"oder-Simpson Theorem)}
  Let X be a topological space, and $\Lambda$ be a lower 
  semi-continuous linear map from $\mathcal V_wX$ to $\exr$.  There is
  a unique lower semi-continuous map $h\in \mathcal LX$ such that
  $\Lambda(\nu) = \int h~d\nu$ for every $\nu\in \mathcal V_wX$, and
  $h(x) = \Lambda (\delta_x)$.
\end{theorem}

\section{The algebras of the extended powerdomain monad}\label{algebras}

\subsection{The algebras of $\V$}

In order to describe the structure maps of the $\V$-algebras, let us
first define barycentres of continuous valuations by imitating a
definition due to \cite[Chapter~6, 26.2]{Choquet69}, and following
\cite{cohen06}.

\begin{definition}
  \label{defn:bary:choquet}
  Let $C$ be a semitopological cone, and $\nu$ be a continuous
  valuation on $C$.  A \emph{barycentre} of $\nu$ is any point
  $b_\nu \in C$ such that, for every linear lower semi-continuous map
  $\Lambda \colon C \to \exr$,
  $\Lambda (b_\nu) = \int \Lambda ~d \nu$.
\end{definition}

\begin{remark}
  \label{rem:barycenter}
  Given a probability measure $\nu$, Choquet called its barycenters
  its \emph{resultants}.  One can also encounter the name \emph{centre
    of gravity}, or \emph{centre of mass}, of $\nu$.  Choquet's
  definition applies to the case where $C$ is a Hausdorff locally
  convex vector space, not a semitopological cone, and uses continuous
  maps $\Lambda$ from $C$ to~$\mathbb R$ with its standard topology, not
  its Scott topology.
\end{remark}

\begin{example}\label{examplebary}
  Let $C$ be a semitopological cone, and
  $\sum_{i=1}^n r_i\delta_{x_i}$ be a simple valuation on~$C$. Then
  $\sum_{i=1}^nr_ix_i$ is a barycentre of
  $\sum_{i=1}^n r_i\delta_{x_i}$. In particular, for any $x\in C$, $x$
  is a barycentre of the Dirac mass $\delta_x$. Indeed, for every lower
  semi-continuous linear function $f \colon C\to \exr$, we have that
  $f(\sum_{i=1}^nr_ix_i)=\sum_{i=1}^nr_if(x_i) = \int
  f~d(\sum_{i=1}^nr_i\delta_{x_i})$.
\end{example}

\begin{example}
  \label{exa:L:barycentre}
  Let $L$ be a complete lattice with its Scott topology and the cone
  structure of Example~\ref{exa:L}.  For every $\nu \in \V L$, the
  \emph{support} $\supp \nu$ is the complement of the largest open set
  $U$ such that $\nu (U)=0$.  (The family of those open sets is
  directed, by the modularity law, and its supremum must be in it, by
  Scott-continuity.)  We claim that the barycentre of $\nu$ is
  $\bigvee \supp L$.  Indeed, using the definition of barycentres and
  the fact that the lower semi-continuous linear maps $\Lambda$ are
  the maps of the form $\infty \cdot \chi_{L \setminus \dc x_0}$,
  $x_0 \in X$, we obtain that $x$ is a barycentre of $\nu$ if and only
  if the following holds: $(*)$ for every $x_0 \in X$, $x \leq x_0$ if
  and only if $\infty \cdot \nu (L \setminus \dc x_0) = 0$.  Since
  $\infty \cdot \nu (L \setminus \dc x_0) = 0$ is equivalent to
  $\supp \nu \subseteq \dc x_0$, hence to the fact that $x_0$ is an
  upper bound of $\supp \nu$, $(*)$ is equivalent to stating that $x$
  is the least upper bound of $\supp \nu$.
\end{example}

\begin{lemma}
  \label{lemma:bary:unique}
  Barycentres on a convex-$T_0$ semitopological cone $C$ are unique
  when they exist.
\end{lemma}
\begin{proof}
  If $x_0$ and $x_1$ are two barycentres of the same continuous
  valuation $\nu$, then $\Lambda (x_0) = \Lambda (x_1)$ for every
  lower semi-continuous linear map $\Lambda \colon C\to \exr$.  Since
  $C$ is convex-$T_0$, $x_0= x_1$.
\end{proof}
  
We now show that the structure maps of the $\V$-algebras are nothing
but maps that send valuations to their barycentres.
  
\begin{lemma}\label{struisbary}
  Let $(X, \alpha)$ be an algebra of the monad $\mathcal V_w$ on the
  category \TOP.  Then $X$ is a topological cone with $+$ defined by
  $x+y = \alpha(\delta_x+\delta_y)$, and scalar multiplication
  defined by $r\cdot x = \alpha(r\delta_x)$ for $r\in \mathbb R_+$ and
  $x, y\in X$. Moreover, the structure map~$\alpha$ is linear and
  sends each $\mu\in \mathcal V_wX$ to a barycentre of~$\mu$.
\end{lemma}
We say that the cone structure obtained this way is \emph{induced} by
the algebra $(X, \alpha)$.  The fact that $\alpha$ is linear and
$\alpha (\mu)$ is a barycentre of $\mu$ has to be understood with
respect to that induced cone structure.  

\begin{proof}
  We first notice that every extension map $f^\dagger$, as given in
  Proposition~\ref{folklore}, is linear, so
  $m_X= \id_{\mathcal V_wX}^\dagger$ and
  $\V \alpha = (\eta_Y\circ \alpha)^\dagger$ are linear.

  Let us show that $X$ with the addition and scalar multiplication
  defined above is a cone.  We only verify the associativity of
  addition and scalar multiplication.  For any $x, y, z\in X$ and
  $r, s\in \mathbb R_+$, we do the following computation:
  \begin{align*}
    (x+y)+z
    &= \alpha (\delta_{\alpha(\delta_x+\delta_y )}+\delta_z)
    & \text{definition of addition on }X
    \\
    &= \alpha (\delta_{\alpha(\delta_x+\delta_y )}+\delta_{\alpha (\delta_z)})
    & \text{definition of structure map} \\
    &= \alpha(\V\alpha (\delta_{\delta_x+\delta_y}) + \V\alpha (\delta_{\delta_z}))
    & \text{naturality of the unit}
    \\
    &= \alpha(\V\alpha (\delta_{\delta_x + \delta_u} + \delta_{\delta_z}))
    & \text{linearity of }\V\alpha
    \\
    &= \alpha m_X(\delta_{(\delta_x+\delta_y)}+ \delta_{\delta_z})
    & \text{definition of structure map}
    \\
    & = \alpha ((\delta_x+\delta_y)+\delta_z)
      & \text{definition of }m_X.
  \end{align*}
  Similarly, $x+(y+z) = \alpha (\delta_x+(\delta_y+\delta_z))$, so
  $(x + y) + z = x + (y + z)$. Moreover,
  \begin{align*}
    r\cdot(s\cdot x)
    &= r\cdot (\alpha(s\delta_x) )
    & \text{definition of scalar multiplication on }X
    \\
    &= \alpha(r\delta_{\alpha(s\delta_x)})
    & \text{definition of scalar multiplication on }X
    \\
    &= \alpha( r\V\alpha (\delta_{s\delta_x}) )
    & \text{naturality of the unit}
    \\
    &= \alpha (\V\alpha (r \delta_{s\delta_x}))
    & \text{linearity of }\V\alpha
    \\
    &= \alpha m_X(r \delta_{s\delta_x})
    & \text{definition of structure map}
    \\
    &= \alpha (rs\delta_x)
    & \text{linearity of }m_X
    \\
    &=(rs)\cdot x
    & \text{definition of scalar multiplication on }X.
  \end{align*}
  To see that $X$ is a topological cone, we assume that $U$ is an open
  set in~$X$ and $x+y\in U$. This means that
  $\alpha(\delta_x+\delta_y)\in U$, hence
  $\delta_x+\delta_y\in \alpha^{-1}(U)$. Since $\mathcal V_w X$ is a
  topological cone and the unit map
  $\eta_X \colon x\mapsto \delta_x \colon X\to \mathcal V_wX$ is
  continuous, we can find open sets $U_x, U_y$ such that
  $x\in U_x, y\in U_y$ and for any $x'\in U_x, y'\in U_y$,
  $\delta_{x'}+\delta_{y'}\in \alpha^{-1}(U)$, which means that
  $x'+y'\in U$ for all $x'\in U_x$ and $y'\in U_y$. This proves that
  $+$ is jointly continuous.  
  The joint continuity of scalar multiplication can be proved
  similarly.
  
  We proceed to prove that $\alpha$ is linear.  Let
  $r\in \mathbb R_+$ and $\mu, \nu\in \V X$.  We have the following:
  \begin{align*}
    \alpha (\mu + \nu)
    & = \alpha (m_X (\delta_\mu) + m_X (\delta_\nu))
    & \text{monad law} \\
    & = \alpha (m_X (\delta_\mu + \delta_\nu))
    & \text{linearity of }m_X \\
    & = \alpha (\V\alpha (\delta_\mu + \delta_\nu))
    & \text{definition of structure map} \\
    & = \alpha (\V\alpha (\delta_\mu) + \V\alpha (\delta_\nu))
    & \text{linearity of }\V\alpha \\
    & = \alpha (\delta_{\alpha (\mu)} + \delta_{\alpha (\nu)})
    & \text{naturality of the unit} \\
    & = \alpha (\mu) + \alpha (\nu)
    & \text{definition of addition on }X.
  \end{align*}
  Similarly, we can prove that $\alpha(r\mu)= r\cdot\alpha(\mu)$.
  
  Finally, we prove that $\alpha(\mu)$ is a barycentre of~$\mu$ for
  all $\mu\in \V X$. Assume that $\Lambda \colon X\to \exr$ is a lower
  semi-continuous linear map. Notice that the composition
  $\Lambda\circ \alpha$ is then a linear map from $\V X$ to
  $\exr$. Hence by the Schr\"oder-Simpson Theorem there exists a
  unique lower semi-continuous map $h \colon X\to \exr$ such that
  $\Lambda\circ \alpha (\nu) = \int h~d\nu$ for all $\nu\in \V X$. In
  particular
  $\Lambda(x)=\Lambda\circ \alpha (\delta_x) = \int h~d\delta_x=h(x)$
  for all $x\in X$. This implies that $h=\Lambda$, and hence
  $\Lambda(\alpha (\mu)) = \int \Lambda ~d\mu$ for all $\mu\in \V
  X$. So $\alpha(\mu)$ is a barycentre of~$\mu$ by definition.
\end{proof}

\begin{corollary}
  Let $X$ be a topological space.  For every
  $\varpi \in \V \V X$, $m_X (\varpi)$ is the barycentre
  $(U\mapsto \int_{\nu \in \V X} \nu(U) d\varpi)$ of $\varpi$ in
  $\V X$.
\end{corollary}
\begin{proof}
  By general category theory, $(\V X, m_X)$ is an algebra of $\V$.
\end{proof}

\begin{corollary} 
  \label{corl:alpha:unique}
  For every convex-$T_0$ semitopological cone $C$, there is at most
  one map $\alpha \colon \V C \to C$ that makes $(C, \alpha)$ a
  $\V$-algebra and induces the original cone structure on $C$.
\end{corollary}
\begin{proof}
  By Lemma~\ref{struisbary}, and since the induced cone structure is
  the original one, $\alpha$ must map every $\nu$ to one of its
  barycentres, and barycentres are unique by
  Lemma~\ref{lemma:bary:unique}.
\end{proof}

\begin{proposition}\label{prop:V:alg:bary}
  Let $(X, \alpha)$ be an algebra of the monad $\mathcal V_w$ on the
  category \TOP. Then $X$ is a weakly locally convex sober topological
  cone with the induced cone structure.
\end{proposition}
\begin{proof}
  By Lemma~\ref{struisbary}, $X$ is a topological cone, and $\alpha$
  is linear.  It is also a continuous retraction by definition of
  algebras, since $\alpha \circ \eta_X = \id_X$. Hence $X$ is linear
  retract of $\V X$, which is locally linear and sober. Since sobriety
  is preserved by continuous retractions, $X$ is a weakly locally
  convex sober topological cone by Proposition~\ref{suls}.
\end{proof}

We may guess that the $\V$-algebras are the sober, weakly locally
convex, topological cones, or maybe those on which, additionally,
every continuous valuation has a barycentre.  This is not quite
enough.  The function $\alpha$ mapping $\nu$ to its barycentre must be
\emph{continuous} as well, and barycentres should be \emph{unique}. 
The latter happens in all convex-$T_0$ cones, but we do not know
whether the cone structure induced by a $\V$-algebra
(Lemma~\ref{struisbary}) is convex-$T_0$.

\begin{proposition}
  \label{prop:V:bary:alg}
  Let $C$ be a semitopological cone, and $\alpha$ be a continuous map
  from $\mathcal V_w C$ to $C$.  If $\alpha (\nu)$ is the unique
  barycentre of $\nu$ for every $\nu \in \mathcal V_wC$, then
  $(C, \alpha)$ is an algebra of the monad $\mathcal V_w$ on the
  category \TOP.

  In that case, the cone structure on $C$ induced by the algebra
  $(C, \alpha)$ coincides with the original cone structure on $C$.
  $C$ is a sober, weakly locally convex, topological cone.
\end{proposition}
\begin{proof}
  For every $x \in C$, $\alpha (\delta_x)=x$ by uniqueness of
  barycentres, and since $x$ is a barycentre of $\delta_x$
  (Example~\ref{examplebary}).  In order to show that
  $\alpha (\V\alpha (\varpi)) = \alpha (m_C (\varpi))$ for every
  $\varpi \in \V\V C$, we consider any lower semi-continuous linear
  function $\Lambda \colon C \to \creal$, and we observe that:
  \begin{align*}
    \Lambda (\alpha (m_C (\varpi)))
    & = \int_{x \in C} \Lambda(x) dm_C (\varpi)
    & \alpha (m_C (\varpi)) \text{ is a barycentre of }m_C (\varpi)
    \\
    & = \int_{x \in C} \Lambda(x) d
      {(\identity {\V C})}^\dagger (\varpi)
    & m_C = {(\identity {\V C})}^\dagger
    \\
    & = \int_{\nu \in \V C} \left(\int_{x \in C} \Lambda(x) d\nu \right)
      d\varpi 
      & \text{disintegration formula (\ref{eq:disint})}\\ 
    & = \int_{\nu \in \V C} \Lambda( \alpha (\nu)) d \varpi
    & \alpha (\nu) \text{ is a barycentre of }\nu
    \\
    & = \int_{x\in C} \Lambda (x)d\V \alpha (\varpi)
    & \text{item (vi) in Lemma~\ref{propertyofint}.}
  \end{align*}
  This shows that $\alpha (m_C (\varpi))$ is also a barycentre of
  $\V\alpha (\varpi)$.  Since barycentres are unique,
  $\alpha (m_C (\varpi)) = \alpha (\V\alpha (\varpi))$.

  Finally, for all $x, y \in C$ we observe that
  $\alpha (\delta_x+\delta_y)$ and $x+y$ are both barycentres of
  $\delta_x+\delta_y$, hence they are equal.  Similarly, for every
  $r \in \Rplus$, $\alpha (r \delta_x) = r \cdot x$.  Hence the
  induced cone structure coincides with the original cone structure on
  $C$.  We conclude by Proposition~\ref{prop:V:alg:bary}.
\end{proof}

This simplifies in the case of locally linear cones, where the
uniqueness of barycentres and the continuity of the barycentre map are
automatic.  
\begin{proposition}
  \label{prop:V:bary:alg:loclin}
  Let $C$ be a locally linear semitopological cone such that every
  continuous valuation $\nu$ on $C$ has a barycentre $b_\nu$.  The
  \emph{barycentre map} $\beta \colon \V C \to C$, defined by
  $\beta (\nu) = b_\nu$, is the structure map of a $\V$-algebra (and
  in particular, $C$ is sober and a topological cone).
\end{proposition}
\begin{proof}
  Since $C$ is locally linear, it is locally convex hence convex-$T_0$
  (Corollary~\ref{coroab}).  Therefore Lemma~\ref{lemma:bary:unique}
  applies, showing that the barycentre $b_\nu$ is unique for every
  $\nu \in \V X$, hence that $\beta$ is well-defined.
  
  We now prove that that $\beta$ is continuous.  Let $H$ be an open
  half-space of $C$; since $C$ is locally linear, it suffices to show
  that $\beta^{-1} (H)$ is open.  If $H=C$, then
  $\beta^{-1} (H) = \V X$ is open.  Otherwise, by
  Theorem~\ref{keimelsep}, there exists a linear lower semi-continuous
  function $h\colon  C\to \exr$ such that $h(a)\leq 1< h(b)$ for all
  $a\in C\setminus H$ and $b\in H$.  Then $H = h^{-1} ((1, \infty])$,
  and $\beta^{-1} (H)$ is then the set of continuous valuations $\nu$
  such that $h (\beta (\nu)) > 1$.  By the definition of barycentres,
  $h (\beta (\nu)) = \int_{x \in C} h(x) d\nu$, so $\beta^{-1} (H)$ is
  equal to the open set $[h > 1]$.  By Proposition~\ref{prop:V:bary:alg},
  $\beta$ is the structure map of a $\V$-algebra.  It follows that $C$
  is sober, and a topological cone, by
  Proposition~\ref{prop:V:alg:bary}.
\end{proof}

\begin{example}
  The extended real numbers $\exr$ with the map
  $\mu\mapsto \int_{x\in \exr}x~d\mu$ is a $\V$-algebra, since $\exr$
  with the Scott topology is a locally linear topological cone.
\end{example}

\begin{example} 
  \label{exa:bary:L}
  Let $L$ be a complete lattice with its Scott topology and the cone
  structure of Example~\ref{exa:L}.
  \begin{enumerate}
  \item If $L$ is a continuous, non-hyper-continuous dcpo (see
    Example~\ref{exa:L:locconv}), then
    $\beta \colon \nu \mapsto \supp \nu$ is the structure map of a
    $\V$-algebra on $L$, as a consequence of the following
    proposition, although $L$ is not locally linear.
  \item If $L$ is not weakly Hausdorff (see below), then $\beta (\nu)$
    is the unique barycentre of $\nu$ for every $\nu \in \mathcal VL$,
    but the cone structure of $L$ is induced by no $\V$-algebra, again
    by the following proposition.  We will also see that every weakly
    Hausdorff complete lattice is sober, hence Isbell's example of a
    non-sober complete lattice \cite{isbell82} is not weakly
    Hausdorff.
  \end{enumerate}
\end{example}
A \emph{weakly Hausdorff} space is a topological space $X$ such that
for all $x, y \in X$, for every open subset $U$ of $X$ that contains
$\upc x \cap \upc y$, there are open neighborhoods $V$ of $x$ and $W$
of $y$ such that $V \cap W \subseteq U$ \cite[Lemma~6.6]{keimel05}.

\begin{proposition}
  \label{prop:bary:L}
  Let $L$ be a complete lattice with its Scott topology and the cone
  structure of Example~\ref{exa:L}.  For every $\nu \in \mathcal VL$,
  let $\beta (\nu) = \bigvee \supp \nu$.  The following are
  equivalent:
  \begin{enumerate}
  \item there is a $\V$-algebra structure on $L$ that induces its cone structure;
  \item $\beta$ is the structure map of a $\V$-algebra on $L$;
  \item $\beta$ is continuous;
  \item $\vee \colon L \times L \to L$ is jointly continuous;
  \item $L$ is weakly Hausdorff.
  \end{enumerate}
  In particular, (i)--(v) hold if $L$ is core-compact, and (i)--(v)
  imply that $L$ is sober.
\end{proposition}
\begin{proof}
  We have seen in Example~\ref{exa:L:barycentre} that $\beta (\nu)$ is
  the barycentre of $\nu$.  This barycentre is unique since $L$ is
  locally convex (Example~\ref{exa:L:locconv}).  The implication
  (iii)$\limp$(ii) then follows from
  Proposition~\ref{prop:V:bary:alg}.  The converse implication is
  trivial.  The equivalence of (i) with (ii) follows from
  Corollary~\ref{corl:alpha:unique}.

  Since every $\V$-algebra is a topological cone
  (Proposition~\ref{prop:V:alg:bary}), (i) implies (iv).
  We now assume (iv), and aim to show (iii).  We will repeatedly use
  the following fact: an open set $U$ intersects $\supp \nu$ if and
  only if $\nu (U) > 0$.  Indeed, $U$ intersects $\supp \nu$ if and
  only if $U$ is not included in the largest open set with
  $\nu$-measure zero.
  
  Let $\nu \in \mathcal VL$ and $V$ be an open neighborhood of
  $\bigvee \supp \nu$.  We will exhibit an open neighborhood $U$ of
  $\nu$ such that $\bigvee \supp \mu \in V$ for every $\mu \in U$.

  Since $V$ is Scott-open, there are finitely many points
  $x_1, \cdots, x_n \in \supp \nu$ whose supremum
  $\bigvee_{i=1}^n x_i$ is in $V$.  By (iv), there are open
  neighborhoods $U_i$ of $x_i$, for each $i$, such that for all
  $y_1 \in U_1$, \ldots, $y_n \in U_n$, $\bigvee_{i=1}^n y_i$ is in
  $V$.  Let us define $U$ as $\bigcap_{i=1}^n [U_i > 0]$.  For every
  $i$, $U_i$ intersects $\supp \nu$ at $x_i$, so $\nu (U_i) > 0$.  It
  follows that $\nu$ is in $U$.  For every $\mu \in U$, we have
  $\mu (U_i) > 0$ for each $i$, so $U_i$ must meet $\supp \mu$, say at
  $y_i$.  Then $\bigvee \supp \mu \geq \bigvee_{i=1}^n y_i \in V$, so
  $\bigvee \supp \mu$ is in $V$.  This shows (iii).

  The equivalence between (iv) and (v) is immediate, since
  $\upc x \cap \upc y = \upc (x \vee y)$.

  As for the last part, the binary supremum operation is jointly
  continuous on any core-compact complete lattice by
  \cite[Corollary~II-4.15]{gierz03}, so if $L$ is core-compact then
  (iv) holds.  If (i) holds, then $L$ is sober by
  Proposition~\ref{prop:V:alg:bary}.
%
\end{proof}

Here is a final example.  
In that case, $\Lform X$ is locally linear (see Example~\ref{exa:LX:locconv}).
\begin{proposition}
  \label{prop:bary:LX}
  For every core-compact space $X$, for every continuous valuation
  $\nu$ on $\Lform X$, the map
  $\beta (\nu) \colon x \in X \mapsto \int_{f \in \Lform X} f (x)
  d\nu$ is the unique barycenter of $\nu$ on $\Lform X$ with its Scott
  topology.  The map $\beta$ is the structure map of a $\V$-algebra on
  $\Lform X$, and the cone structure it induces on $\Lform X$ is the
  usual one.
\end{proposition}
\begin{proof}
  Kirch characterised the algebras of the $\mathcal V$ monad on the
  category \textbf{CONT} of continuous dcpos: the Eilenberg-Moore
  category of $\mathcal V$ on \textbf{CONT} is equivalent to the
  category of continuous d-cones \cite[Satz~7.1]{kirch93}.  A d-cone
  is a dcpo with a cone structure whose addition and scalar
  multiplication are Scott-continuous.  Every continuous d-cone is a
  topological cone, as a consequence of Ershov's theorem (see
  Remark~\ref{rem:ershov}, and
  \cite[Corollary~6.9~(c)]{Keimel:topcones2}).  Since $X$ is
  core-compact, $\Lform X$ is a continuous d-cone, hence there is a
  map $\beta \colon \mathcal V (\Lform X) \to \Lform X$ that turns
  $(\Lform X, \beta)$ into a $\mathcal V$-algebra.  Using another
  result of Kirch \cite[Satz~8.6]{kirch93}, which states that for
  every continuous dcpo $Y$ in its Scott topology, the Scott and weak
  topologies agree on $\mathcal VY$, $\beta$ is a $\V$-algebra map.
  The value $\beta (\nu)$ is given as a directed supremum of
  barycentres of simple valuations way-below $\nu$, and one can check
  that this is equal to $\int_{f \in \Lform X} f (x) d\nu$.  Here is
  an alternative proof, which the reader may find interesting.
  
  We will use Jones' version of Fubini's theorem
  \cite[Theorem~3.17]{jones90}.  This states that for every (jointly)
  continuous map $f \colon X \times Y \to \creal$, where $\creal$ has
  the Scott topology, for all continuous valuations $\mu$ on $X$ and
  $\nu$ on $Y$,
  \[
    \int_{x \in X} \left(\int_{y \in Y} f (x, y) d\nu\right) d\mu
    = \int_{(x,y) \in X \times Y} f (x, y) d (\mu \times \nu)
    = \int_{y \in Y} \left(\int_{x \in X} f (x, y) d\mu\right) d\nu,
  \]
  where $\mu \times \nu$ is the uniquely determined \emph{product
    valuation}.  Implicit in that theorem is the following fact: $(*)$
  the maps $y \in Y \mapsto \int_{x \in X} f (x, y) d\mu$ and
  $x \in X \mapsto \int_{y \in Y} f (x, y) d\nu$ are lower
  semi-continuous.  The latter can be shown using step functions as
  in the proof of Proposition~\ref{folklore}.

  Let $g (x) = \int_{f \in \Lform X} f (x) d\nu$.  This makes
  sense because the map $f \in \Lform X \mapsto f (x)$ is
  Scott-continuous, hence lower semi-continuous, for every $x \in X$.
  In order to show that $g$ is continuous, we first notice that
  $\App \colon \Lform X \times X \to \creal$, which maps $(f, x)$ to
  $f (x)$, is (jointly) continuous.  Indeed, $\Lform X$ is a
  continuous lattice (see Example~\ref{exa:LX:top}), hence a c-space.
  $\App$ is clearly separately continuous, and then jointly continuous
  by Ershov's theorem (see Remark~\ref{rem:ershov}).  By $(*)$, the
  map $x \in X \mapsto \int_{f \in \Lform X} \App (f, x) d\nu$ is
  lower semi-continuous.  But that map is simply $g$.

  Now that $g$ is in $\Lform X$, we check that it is a barycenter of
  $\nu$.
  Let $\Lambda$ be any lower semi-continuous function from
  $\Lform X$ to $\creal$.  $\Lambda$ is integration with respect to
  some uniquely defined continuous valuation $\mu$ on $X$, by the
  Riesz-type representation theorem mentioned earlier.  Then:
  \begin{align*}
    \Lambda (g)
    & = \int_{x \in X} g (x) d\mu \\
    & = \int_{x \in X} \left(\int_{f \in \Lform X} \App (f, x) d\nu\right)
      d\mu \\
    & = \int_{f \in \Lform X} \left(\int_{x \in X} \App (f, x) d\mu\right)
      d\nu
    & \text{by Jones's version of Fubini's theorem} \\
    & = \int_{f \in \Lform X} \Lambda (f) d\nu,
  \end{align*}
  showing that, indeed, $g$ is a barycenter of $\nu$.  Note that
  Jones' version of Fubini's theorem applies, crucially, because
  $\App$ is \emph{jointly} continuous.

  There is nothing more to prove: we merely apply
  Proposition~\ref{prop:V:bary:alg:loclin}, using the fact that
  $\Lform X$ is locally linear.
%
%
\end{proof}


\subsection{The morphisms of algebras of $\V$}

Adapting the definition of morphisms of algebras in our setting, a
morphism $f$ between two $\V$-algebras $(X, \alpha)$ and $(Y,\beta)$
is a continuous function $f\colon X\to Y$ such that
$\beta\circ \V f = f\circ \alpha$. Considering the cone structure of
$X$ and $Y$ induced by $\alpha $ and $\beta$
(Proposition~\ref{prop:V:alg:bary}), respectively, we will see that
$f$ is a linear map between $X$ and $Y$.  Indeed, for any $a, b\in X$
and $r\in \mathbb R_+$,
\begin{align*}
  f(a+b)
  & =f(\alpha(\delta_a+\delta_b))
  & \text{definition of addition} \\
  &= \beta(\V f(\delta_a+\delta_b))
  & \text{$f$ is a morphism of algebras}
  \\
  & = \beta(\delta_{f(a)}+\delta_{f(b)})
  & \text{naturality of the unit}
  \\
  &= f(a)+f(b)
  & \text{definition of addition.} \\
\end{align*}
Similarly, we have $f(r\cdot a)=r\cdot f(a)$. Conversely, we want to
know whether continuous linear maps are exactly the $\V$-algebra
morphisms. To prove this, however, we need to assume that $Y$ is
convex-$T_0$, and not just weakly locally convex.

\begin{proposition}
  Let $(X, \alpha), (Y,\beta)$ be two $\V$-algebras, viewed as
  topological cones in the sense of Proposition~\ref{prop:V:alg:bary}.
  If $Y$ is convex-$T_0$, then the $\V$-algebra morphisms from
  $(X,\alpha)$ to $(Y,\beta)$ are precisely the continuous linear maps
  between them.
\end{proposition}
\begin{proof}
  Let $f \colon X\to Y$ be a linear map. For every lower
  semi-continuous linear map $\Lambda \colon Y\to \exr$,
  $\Lambda \circ f \colon X\to \exr$ is lower semi-continuous and
  linear. Since structure maps send valuations to their barycentres by
  Lemma~\ref{struisbary}, for any continuous valuation $\mu\in \V X$
  we have,
  \begin{align*}
    \Lambda (f(\alpha (\mu)))
    & = (\Lambda \circ f) (\alpha (\mu)) \\
    & = \int \Lambda\circ fd\mu
    & \alpha(\mu) \text{ is a barycentre of }\mu
    \\
    &= \int \Lambda ~d(\V f(\mu))
    & \text{item (vi) in Lemma~\ref{propertyofint}}
    \\
    &= \Lambda (\beta( \V f(\mu)))
    & \beta( \V f(\mu)) \text{ is a barycentre of }\V f(\mu).
  \end{align*}
  Since $Y$ is convex-$T_0$, we use Corollary~\ref{coroab} to conclude
  that $f(\alpha (\mu)) =\beta( \V f(\mu))$.
\end{proof}

\section{The algebras of $\VF$ and $\VP$}\label{fandp}

Besides the space $\V X$ of continuous valuations on any topological
space~$X$, Heckmann also considered its subspaces $\VF X$ of simple
valuations and $\VP X$ of point-continuous valuations
on~$X$~\cite{heckmann96}.  In the same paper he showed that $\VP X$ is
the sobrification of $\VF X$~\cite[Theorem 5.5]{heckmann96}. We will
see that $\VF$ and $\VP$ can also be extended to monads on the
category $\ctop$.

We have seen simple valuations in~Example~\ref{finitevaluations}. We
proceed to define point-continuous valuations.

For a topological space $X$, according to Heckmann~\cite{heckmann96},
one considers, instead of the Scott topology, the \emph{point
  topology} on~$\mathcal OX$ determined by the subbasic open sets
$\mathcal O(x)$, $x\in X$, where
$\mathcal O(x) = \set{U\in \mathcal OX}{x\in U}$ for each $x\in X$.
We denote $\mathcal OX$ with the point topology by
$\mathcal O_{\mathrm p}X$.  One can equate $\Open X$ with the set of
continuous maps from $X$ to Sierpi\'nski space $\Sierp = \{0, 1\}$
(with the Scott topology of $\leq$), and then the point topology is
the subspace topology induced by the inclusion into $\Sierp^X$.

\begin{definition}
  A \emph{point-continuous valuation} $\mu$ on $(X, \mathcal OX)$ is a
  valuation that is continuous from $\mathcal O_{\mathrm p}X$
  to~$\exr$. The set of all point-continuous valuations on $X$ is
  denoted by $\VP X$.
\end{definition}

One easily sees that every simple valuation is point-continuous, and
every point-continuous valuation is a continuous valuation, since the
Scott topology on~$\mathcal OX$ is finer than the point topology.

In what follows, we consider $\VF X$ and $\VP X$ as subspaces
of~$\V X$, that is, the topologies considered are the subspace
topologies induced from the weak topology on~$\V X$.

\begin{proposition}
  Let $f\colon  X\to Y$ be a continuous function between topological
  spaces~$X$ and~$Y$. Then the map
  $\mathcal V_* f\colon \mu\mapsto (U\in \mathcal OY\mapsto \mu(f^{-1}(U)))$
  is continuous from $\mathcal V_*X$ to $\mathcal V_*Y$, where
  $\mu\in \mathcal V_* X$ and $*$ is ${\mathrm p}$ or ${\simple}$.
\end{proposition}
 \begin{proof}
   The only difficult point is to show that $\mathcal V_{\mathrm p} f$
   sends point-continuous valuations to point-continuous
   valuations. To this end, let $\mu$ be any point-continuous
   valuation and $U$ be any open subset in $Y$, and $r$ be any
   positive number in~$\mathbb R_+$ with
   $\mathcal V_{\mathrm p} f(\mu)(U)> r$.  By definition,
   $\mu (f^{-1}(U))>r$.  Since $\mu$ is point-continuous, we can find
   a finite subset $F$ of points such that
   $f^{-1}(U)\in \bigcap_{x\in F}\mathcal O(x)$ and for every open
   subset $V\in \bigcap_{x\in F}\mathcal O(x)$, $\mu(V)>r$.  We claim
   that $\bigcap_{y\in f(F)}\mathcal O(y)$ is an open set containing
   $U$ and such that, for every
   $W\in \bigcap_{y\in f(F)}\mathcal O(y)$, $\VP f(\mu)(W)>r$. The
   former is obvious since
   $f^{-1}(U)\in \bigcap_{x\in F}\mathcal O(x)$ means that
   $F\subseteq f^{-1}(U)$, i.e., $f(F)\subseteq U$.  For the latter
   claim, we know that $f(F)\subseteq W$, so we have
   $f^{-1}(W)\in \bigcap_{x\in F} \mathcal O(x)$.  From the
   point-continuity of $\mu$, we have $\mu (f^{-1}(W))>r$, hence
   $\mathcal V_{\mathrm p} f(\mu)(W)>r$.
   
   $\VP f$ is continuous since $\V f$ is
   (Proposition~\ref{prop:V:functor}).
 \end{proof}

 \begin{remark}
   \label{rem:VF:f}
   The following formula holds:
   $\VF f (\sum_{i=1}^n a_i \delta_{x_i}) = \sum_{i=1}^n a_i \delta_{f
     (x_i)}$.
 \end{remark}

\begin{proposition}
  For all topological spaces $X$ and $Y$, and for every continuous
  function $f \colon X\to \mathcal V_* Y$, the map
  $$f_*^\dagger \colon \mu\mapsto \left(U\mapsto \int_{x\in X}f(x)(U)d\mu\right)
  \colon \mathcal V_*X \to \mathcal V_*Y$$ is well-defined and
  continuous, where $*$ is ${\mathrm p}$ or ${\simple}$.
\end{proposition}
\begin{proof}
  If $f_*^\dagger$ indeed takes its values in $\mathcal V_*Y$, then it
  is continuous, because $f^\dagger$ is---that is part of
  Proposition~\ref{folklore}.
  

  We proceed to prove that $f_*^\dagger$ takes its values in
  $\mathcal V_*Y$.  When $*=\mathrm{f}$, we assume that
  $\mu=\sum_{i=1}^n r_i\delta_{x_i}$. Then
  $f_{\simple}^\dagger(\mu)(U)=\sum_{i=1}^n r_if(x_i)(U)$. Since for
  each~$i\in I$, $r_if(x_i)$ is a simple valuation,
  $f_{\simple}^\dagger(\mu)$, as a finite sum of simple valuations,
  is again a simple valuation.

  We now show that $f_{\mathrm p}^\dagger$ takes its values in
  $\mathcal V_{\mathrm p} Y$.  In order to see this, we first notice
  that $f_{\simple}^\dagger$ is also a continuous map from $\VF X$
  to $\VP Y$, considering $\VF Y$ as a subspace of $\VP Y$. Since
  $\VP X$ is the sobrification of $\VF X$ \cite[Theorem
  5.5]{heckmann96}, the function $f_{\simple}^\dagger$ has a unique
  continuous extension $e$ from $\VP X$ to~$\VP Y$. Considering
  $\VP Y$ as a subspace of $\V Y$, then from
  Proposition~\ref{folklore} we know that both $e$ and
  $f_{\mathrm p}^\dagger$ are continuous functions from $\VP X$ to
  $\V Y$. Since $\V Y$ is $T_0$, and $e$ and $f_{\mathrm p}^\dagger$
  coincide on~$\VF X$, they coincide on the sobrification~$\VP X$ as
  well. Thus $f_{\mathrm p}^\dagger$ sends point-continuous valuations
  to point-continuous valuations since $e$ does.
\end{proof}

With all the ingredients listed above, we conclude the following:

\begin{proposition}
  $\mathcal V_*$ ($*$ is $\mathrm p$ or $\simple$) is a monad on the
  category $\ctop$, with the unit
  $\eta_* \colon x\to \delta_x\colon  X\to \mathcal V_*X$ and extension
  $$f_*^\dagger \colon \mu\mapsto (U\mapsto \int_{x\in X}f(x)(U)d\mu) \colon
  \mathcal V_*X \to \mathcal V_*Y$$ for continuous map
  $f\colon  X\to \mathcal V_* Y$. The multiplication $m^*_X$ of
  $\mathcal V_*$ at $X$ is $(\id_{\mathcal V_*X})_*^\dagger$. \qed
\end{proposition}

Similarly to $\V X$, $\VF X$ and $\VP X$ are also locally linear
topological cones with the canonical operations of addition and scalar
multiplication. Moreover, we have the following:

\begin{theorem}{\rm \cite[Theorem 6.7, Theorem 6.8]{heckmann96}}\label{twomains}
  \begin{enumerate}
  \item $\VF X$ is the free weakly locally convex cone over~$X$ in the
    category $\ctop$.
  \item $\VP X$ is the free weakly locally convex sober cone over~$X$
    in the category $\ctop$.
  \end{enumerate}
\end{theorem}
This means that for a $T_0$ space~$X$, $\VF X$ (resp., $\VP X$) is a
weakly locally convex (resp., weakly locally convex sober) topological
cone, and for every continuous function $f\colon  X\to C$ from $X$ to a
weakly locally convex (resp., weakly locally convex sober) topological
cone~$C$, there is a unique continuous linear function
$\overline{f}\colon  \mathcal V_* X\to M$ such that
$\overline{f}\circ \eta_*=f$, where $*$ is $\simple$ or $\mathrm p$.

The following is a straightforward consequence of the above theorem.

\begin{corollary}
  A topological cone $C$ is weakly locally convex (resp., weakly
  locally convex and sober) if and only if it is a continuous linear
  retract of a locally linear (resp., locally linear and sober)
  topological cone.
\end{corollary}
\begin{proof}
  We have seen the ``if'' direction in~Proposition~\ref{suls}. 

  For the ``only if'' direction, since $\id_C$, the identity map
  over~$C$, is continuous and linear, there exists a unique continuous
  linear map $\overline{\id_C}\colon  \VF C\to C$ such that
  $\overline{\id_C}\circ\eta_{\simple}=\id_C$. This exhibits that
  $C$ is a linear retract of~$\VF C$ which is a locally linear
  topological cone.

  To show that every weakly locally convex sober topological cone is a
  linear retract of some locally linear sober topological cone, we
  just change $\VF$ into $\VP$ in the above and the same argument
  applies.
\end{proof}

The following results from~\cite[Theorem 6.1]{heckmann96} are needed
for our further discussion.
\begin{theorem}\label{uniqueondm}
  Let $X$ be a topological space. 
  \begin{enumerate}
  \item Every linear function from $\VF X$ to some cone is uniquely
    determined by its values on Dirac masses.
  \item Every continuous linear function from $\VP X$ to some
    topological cone is uniquely determined by its values on Dirac
    masses.
  \end{enumerate}
\end{theorem}

We now have enough ingredients to prove the main results of this
section.

\begin{theorem}\label{theoremvf}
  Let $X$ be a $T_0$ topological space, and $\alpha \colon \VF X\to X$
  be a continuous map. If $(X, \alpha)$ is a $\VF$-algebra, then $X$
  is a weakly locally convex topological cone, and $\alpha$ is the
  \emph{standard barycentre} map
  $\sum_{i=1}^n r_i \delta_{x_i} \mapsto \sum_{i=1}^n r_i
  x_i$. Conversely, for every weakly locally convex topological
  cone~$C$, there exists a (unique) continuous linear map~$\alpha$
  from $\VF C$ to $C$, sending each simple valuation to its
  barycentre, and the pair $(C, \alpha)$ is a $\VF$-algebra.
\end{theorem}
\begin{proof}
  A similar argument as in the proof of Lemma~\ref{struisbary} shows
  that $X$ is a topological cone with $+$ defined by
  $x+y = \alpha(\delta_x+\delta_y)$, and scalar multiplication
  defined by $r\cdot x = \alpha(r\delta_x)$ for all $r\in \mathbb R_+$
  and $x, y\in X$. $X$ is weakly locally convex since $\VF X$ is
  locally linear and the structure map $\alpha$ is a linear
  retraction.  We show that $\alpha$ sends each simple valuation to
  its barycentre. This is easy since $\alpha(\delta_x)= x$ by the
  definition of structure map, hence by linearity $\alpha$ sends every
  simple valuation $\sum_{i=1}^n r_i\delta_{x_i}$ to
  $\sum_{i=1}^n r_ix_i$, which is a barycentre of
  $\sum_{i=1}^n r_i\delta_{x_i}$ (Example~\ref{examplebary}).

  Conversely, assume that $C$ is a weakly locally convex topological
  cone. Since $\VF C$ is the free weakly locally convex topological
  cone over~$C$, there exists a unique map $\alpha$ such that
  $\alpha \circ \eta_{\simple}=\id_C$. Hence $\alpha (\delta_x)= x$
  for every $x\in C$. Then $\alpha$ sends each simple valuation to its
  barycentre, since $\alpha$ is linear. Finally, to see that
  $(C, \alpha)$ is a $\VF$-algebra, we only need to verify that
  $\alpha\circ m^{\simple}_C = \alpha\circ \VF \alpha_C$. Notice
  that both sides of the equals sign are continuous linear functions
  from $\VF \VF C$ to $C$. From Theorem~\ref{uniqueondm} we only need
  to show they are equal on Dirac masses. To this end, let us assume
  that $\mu$ is a simple valuation on $X$, and we compute the
  following,
  \begin{align*}
    \alpha \circ m^{\simple}_C (\delta_\mu )
    &= \alpha (\mu)
    & \text{monad law}
    \\
    &= \alpha(\delta_{\alpha(\mu)})
    & \alpha \circ \eta = \id 
    \\
    &= \alpha\circ \VF \alpha(\delta_{\mu})
    & \text{naturality of the unit,}      
  \end{align*}
  and this concludes the proof.
\end{proof}

The $\VF$-algebra morphisms are precisely the continuous linear maps between them.

\begin{theorem}\label{morphismvf}
  Let $X, Y$ be two weakly locally convex topological cones, and
  $\alpha\colon \VF X\to X, \beta\colon\VF Y\to Y$ be the
  corresponding barycentre maps. Then a continuous map $f\colon  X\to Y$ is
  a $\VF$-algebra morphism if and only if $f$ is linear.
\end{theorem}
\begin{proof}
  Assume first that $f$ a $\VF$-algebra morphism. For any $a, b\in X$,
  $r\in \mathbb R_+$, we have
  \begin{align*}
    f(a+b)
    &= f(\alpha(\delta_a+\delta_b))
    & \alpha \text{ is the barycentre map}
    \\
    &= \beta(\VF f(\delta_a+\delta_b))
    & f \text{ is a morphism of algebras}
    \\
    &= \beta(\delta_f(a)+\delta_f(b))
    & \text{Remark~\ref{rem:VF:f}}
    \\
    &= f(a)+f(b)
    & \alpha \text{ is the barycentre map}
  \end{align*}
  Similarly, we can prove that $f(ra)=rf(a)$.

  Conversely, assume $f$ is linear. Then for any simple valuation
  $\sum_{i=1}^n r_i\delta_{x_i}$, we have:
  \begin{align*}
    f(\alpha(\sum_{i=1}^n r_i\delta_{x_i}))
    &=  f(\sum_{i=1}^n r_ix_i)
    & \alpha \text{ is the barycentre map}
    \\
    & = \sum_{i=1}^n r_if(x_i)
    & f \text{ is linear}, 
  \end{align*}
  and 
  \begin{align*}
    \beta(\VF f(\sum_{i\in I}r_i\delta_{x_i}))
    & = \beta(\sum_{i\in I}r_i\delta_{f(x_i)})
    & \text{Remark~\ref{rem:VF:f}}
    \\
    & =\sum_{i\in I} r_if(x_i)
    & \beta \text{ is the barycentre map.}      
  \end{align*}
  Hence $f\circ \alpha = \beta\circ \VF f$, and therefore $f$ is a
  $\VF$-algebra morphism.
\end{proof}

In weakly locally convex topological cones, even simple valuations may
have several barycentres.  As in Theorem~\ref{theoremvf}, we call $x$
the \emph{standard barycentre} of $\delta_x$.  Note that this is
well-defined: if $\delta_x = \delta_y$, then for every open set $U$,
$\delta_x (U)=1$ if and only if $\delta_x (U)=1$, hence $x$ and $y$
have the same open neighbourhoods, which implies $x=y$ since the cone
is $T_0$.
\begin{theorem}
  Let $X$ be a $T_0$ topological space, and $\alpha \colon \VP X\to X$
  be a function. If $(X, \alpha)$ is a $\VP$-algebra, then $X$ is a
  weakly locally convex sober topological cone, and $\alpha$ is a
  linear continuous function that maps every point-continuous
  valuation to one of its (Choquet) barycentres, and every Dirac mass
  to its standard barycentre.

  Conversely, for every weakly locally convex sober topological
  cone~$C$, there exists a (unique) continuous linear map~$\alpha$
  from $\VP C$ to~$C$ that sends each Dirac mass to its standard 
  barycentre.  Then the pair $(C, \alpha)$ is a $\VP$-algebra.
\end{theorem}
\begin{proof}
  The same reasoning as in the proof of~Theorem~\ref{theoremvf} will
  show that $X$ is a continuous linear retract of~$\VP X$. Since
  $\VP X$ is locally linear and sober, then $X$, as a continuous
  linear retract, is weakly locally convex sober.  The proof that
  $\alpha$ is linear is as in Lemma~\ref{struisbary}.

  To see that $\alpha$ maps every point-continuous valuation $\mu$ to
  one of its barycentres, for any continuous linear map
  $\Lambda \colon X\to \exr$ we consider two maps from $\VP X$ to $\exr$.
  They are $\mu\mapsto \Lambda\circ \alpha(\mu)$ and
  $\mu\mapsto \int \Lambda ~d\mu$. Note that these two maps are
  continuous linear maps and coincide on Dirac masses on~$X$, hence
  they are equal from Item~(ii) of Theorem~\ref{uniqueondm}.  Hence
  $\alpha$ is indeed a barycentre map.

  Conversely, for any weakly locally convex sober cone~$C$, let $f$ be
  the identity map from $C$ to $C$.  By Item~(ii) of
  Theorem~\ref{twomains}, there is a unique continuous linear map
  $\overline{f}$ such that $\overline{f} \circ \eta_{\mathrm p} = f$,
  and this is the desired $\alpha$.  That $(C, \alpha)$ is a
  $\VP$-algebra can be verified similarly as in the proof
  of~Theorem~\ref{theoremvf}.
\end{proof}

\begin{theorem}
  Let $X, Y$ be two weakly locally convex sober topological cones, and
  $\alpha\colon \VP X\to X, \beta\colon\VP Y\to Y$ be the structure
  maps of the corresponding $\VP$-algebras. Then a continuous map
  $f \colon X\to Y$ is a $\VP$-algebra morphism if and only if $f$ is
  linear.
\end{theorem}
\begin{proof}
  First, if $f$ is a $\VP$-algebra morphism, then that $f$ is linear
  follows from the same argument as in the proof
  of~Theorem~\ref{morphismvf}.

  Conversely, if $f$ is a continuous linear map, then we have already
  seen in~Theorem~\ref{morphismvf} that
  $f\circ \alpha (\delta_x)= \beta\circ \VP f(\delta_x)$ for any Dirac
  mass $\delta_x$. Since $\alpha$ and $\beta$ are structure maps, they
  are linear.  From its definition, $\VP f$ is linear.  Hence both
  $f\circ \alpha$ and $\beta\circ \VP f$ are continuous linear maps.
  Using Item~(ii) of Theorem~\ref{uniqueondm}, they coincide
  on~$\VP X$.
\end{proof}

\bibliographystyle{entcs}
\ifentcs
\bibliography{valg}
\else

\fi

\end{document}